\documentclass[]{interact}

\usepackage{wrapfig}
\usepackage{epstopdf}
\usepackage[caption=false]{subfig}

\usepackage[longnamesfirst,sort]{natbib}
\bibpunct[, ]{(}{)}{;}{a}{,}{,}

\theoremstyle{plain}
\newtheorem{theorem}{Theorem}[section]

\theoremstyle{definition}
\newtheorem{definition}[theorem]{Definition}

\theoremstyle{remark}
\newtheorem{remark}{Remark}

\begin{document}


\title{ On differential inclusions arising from some discontinuous systems }

\author{
\name{A.~V. Fominyh \textsuperscript{a,}\textsuperscript{b}\thanks{The main results of this paper (sections 5--7) were obtained in IPME RAS and supported by Russian Science Foundation (project no. 20-71-10032).}}
\affil{\textsuperscript{a}Institute for Problems in Mechanical Engineering, Russian Academy of Sciences, Russia
\textsuperscript{b}Faculty of Applied Mathematics and Control Processes, {Saint Petersburg State University}, {Saint Petersburg}, {Russia}}
}

\maketitle

\begin{abstract}
The paper deals with systems of ordinary differential equations containing in the right-hand side controls which are discontinuous in phase variables. These controls cause the occurrence of sliding modes. If one uses one of the well-known definitions of the solution of discontinuous systems, then the motion of an object while being on some surface can be described in terms of differential inclusions. With the help of the previously developed apparatus for solving differential inclusions, a method is constructed for finding the trajectories of a system moving in a such a mode. Since some of frequently used discontinuous controls contain nonsmooth functions of phase variables, the paper pays special attention to study the differential properties of such systems. At the end of the paper controls of a slightly different, in contrast to the classical, type are considered which have useful differential properties, and a method is constructed for solving systems with such controls considered both before hitting the required surface and moving in its vicinity.
\end{abstract}


\section{Introduction}

In many practical problems, when trying to construct mathematical models of real physical processes, discontinuous control is used; therefore, the right-hand sides of the systems of differential equations describing the process under study are discontinuous functions of the state vector. For example, the control may be an $m$-dimensional vector function of the following form 
\begin{equation}
\label{0}
u_i(x,t) = 
\left\{
\begin{array}{ll}
u_i^+(x,t), \ s_i(x) > 0, \\
u_i^-(x,t), \ s_i(x) < 0,
\end{array}
\right.
\end{equation}
where $t$ belongs to the time interval on which the motion of the system is considered, the functions $u_i^+(x,t)$, $u_i^-(x,t)$, $i = \overline{1,m}$, are continuous, and $s_i(x) = 0$ are discontinuity surfaces  ($s_i(x)$, $i = \overline{1,m}$, are some continuous functions of the phase vector $x$). Put $s(x) = (s_1(x), \dots s_m(x))'$.
Those cases, in which the discontinuity points are isolated, belong to classical control theory and have been studied more fully than the cases where the set of discontinuity points constitutes a set of nonzero measure in time. 

In such systems the state vector may stay on one of the discontinuity surfaces (or on some their intersection) over a period of time of nonzero measure. Motion along the discontinuity surfaces or along some their intersection is called a sliding mode.  Since the system trajectories belonging to the set of discontinuity points do not coincide with any of the system trajectories resulting from various combinations of continuous controls $u_i^+(x,t)$, $u_i^-(x,t)$, $i = \overline{1,m}$, they are singular. Therefore it is required to introduce a corresponding definition in case when the sliding mode occurs. Let us discuss one of the classical variants of such a definition (see \cite{Filippov}, \cite{AizerPyatn}) which will be used in this paper. 
 On the finite time interval $[0, T]$ consider the system 
 $\dot x = f(x, u_1(x, t), \dots, u_m(x, t), t)$, 
 in which the vector-function $f(x, u_1, \dots, u_m, t)$ is continuous in all its arguments, and the vector-functions $u_i (x, t)$, $i = \overline{1,m}$, 
 are discontinuous on the sets $s_i(x) = 0$, $i = \overline{1,m}$, respectively. 
 At every point $(x, t)$ of discontinuity of the vector-function $u_i (x, t)$, $i = \overline{1,m}$, 
 a closed set $U_i (x, t)$, $i = \overline{1,m}$, must be defined. It is a set of possible values of the variable  
 $u_i$ of the function $f(x, u_1, \dots, u_m, t)$. 
 Denote $F (x, t) = f(x, u_1, ..., u_m, t)$  
the set of the function $f(x, u_1, ..., u_m, t)$ values at the fixed variables 
 $x, t$, and while $u_1, \dots, u_m$ run through the sets 
 $U_1(x, t), \dots, U_m(x, t)$ respectively. Then the solutions of this differential inclusion are taken as solutions of the original differential equation with a discontinuous right-hand side. 
In physical systems the sets $U_i (x, t)$, $i = \overline{1,m}$, usually correspond to different blocks and are assumed convex. At every point $(x, t)$ of discontinuity of the vector-function $u_i (x, t)$, $i = \overline{1,m}$, the set $U_i (x, t)$ must also contain all the limit points of all sequences $v_k \in U_i (x_k, t_k)$, where $x_k \rightarrow x$ and $t_k \rightarrow t$ if $k \rightarrow \infty$. If the control is of form~(\ref{0}), then it is natural to consider $U_i(x, t) = \mathrm{co} \{u_i^-(x,t), u_i^+(x,t)\}$, $i = \overline{1,m}$, as such sets. Note that there exist some other common definitions of discontinuous systems solutions (see, e. g., \cite{Filippov}, \cite {Utkin}). 
 
 It was noticed \cite{Filippov}, \cite{Utkin} that when the system state is in a sliding mode, new useful properties are observed that are absent when the system moves in a ``normal'' mode. For example, it is precisely such motions that are optimal in the sense of some criterion in the optimal control theory. The artificial introduction of a sliding mode into the system is also used in order to stabilize it, as well as to get rid of unwanted disturbances acting on the object under consideration. The sliding modes have numerous applications. For example, in work \cite{AshrafiuonMuske} a sliding mode is used for trajectory tracking of autonomous surface vessels. In paper \cite{Beltran} a sliding mode was used in order to stabilize a wind turbine in spite of model uncertainties. Herewith, in the most of the papers listed both first-order and second-order sliding surfaces are implemented. It was noted that higher-order sliding modes preserve or generalize the main properties of the standard ones and mitigate dangerous chattering effect, so they are also widely used in the literature (see, e. g., \cite{Levant}, \cite{EmelKorLev}). In some papers (\cite{Borta}, \cite{Furuta}) sliding modes were applied to discrete systems; herewith the control designs were developed mitigating undesirable chattering and high-frequency switching between different values of the control signal. Note also paper \cite{XuOzguner} where a nontrivial and an effective control design is constructed for a class of underactuated systems. Some of the papers aim at constructing a continuous control (see section 7), for example, \cite{Tang} (in application to rigid robots), and also \cite{Shtessel} (for the missile-interceptor guidance system against targets performing evasive maneuvers) which reduce control efforts in the transient state.  
 
 Therefore, it makes sense to pose the problem of choosing a control that after a finite time interval transfers the system trajectory from an arbitrary initial state to a small neighborhood of the discontinuity surface in which the system continues to move for the rest of time. This is the essence of providing the stability of the sliding mode of the system. Moreover, if one additionally requires the fulfillment of the condition $s(x(t)) \rightarrow {\bf 0_m}$ at $t \rightarrow \infty$ (where $s(x)~=~{\bf 0_m}$ is a discontinuity surface), then it is the stability ``in big''. The given definition is of an qualitative nature; a strict definition of the sliding mode stability may be found in \cite{Utkin}. A very powerful tool for studying the stability of sliding modes is the theory of Lyapunov and Barbashin--Krasovskii which originates from classical stability theory. In this paper a feedback control will be considered for the class of systems that endows the Lyapunov function with the required properties, thereby ensuring the stability of sliding modes.
 
As one will see from the strict statement of the problem, the system motion along the discontinuity surface can be described by some interval system. As is shown in work \cite{Fominyh1}, an interval system may be written down in the form of a differential inclusion of a certain structure. So, if we use the above definition of the discontinuous system solution, then after hitting the discontinuity surface (we recall that the control used sets itself this goal, as well as the goal of keeping the trajectory on this surface), the system trajectory motion is described by a differential inclusion. 
 
 This paper aims at the problem of finding a solution of a differential inclusion describing the system motion in a sliding mode. Herewith, we will additionally impose some restrictions on the desired trajectory, for example, hitting a certain point at the final moment of time (when moving along the discontinuity surface). For solving this problem we will use the previously developed apparatus for solving interval systems considered as differential inclusions \cite{Fominyh1}, \cite{Fominyh2}, \cite{Fominyh3}. As will be seen below, this inclusion in some cases has the following feature: due to the special control structure the right-hand sides of the differential inclusion can be nonsmooth functions of the phase coordinates. Therefore, the paper also discusses the conditions for the appearance of nonsmoothness in the right-hand side of a differential inclusion and investigates the differential properties of such systems. Finally, in the final part of the paper we discuss a somewhat different, in contrast to the classical, feedback control structure which allows one to preserve some useful differential properties of the right-hand side of the system, while at the same time (just like the classical control type) ensuring that the system gets into a small enough neighborhood of the required surface; although in the general case it does not ensure the stability ``in big'' of the sliding mode. The control structure proposed allows one to develop a method for searching for the system trajectory moving both in the vicinity of this surface and before hitting it. Sometimes, in order to fulfill some restrictions on the trajectory, it is necessary to choose an appropriate structure for the discontinuous surface itself; such problems are called the sliding mode design. From physical considerations approximate structure and parameters are often known, so it makes sense to ``correct'' these parameters in order to bring a system trajectories the desired properties (see Remark~2 below); this problem is also being solved in this paper.

\section{Basic definitions and notations}
In the paper we will use the following notations. $C_{n} [0, T]$ is the space of $n$-dimensional continuous on $[0, T]$ vector-functions. $P_{n} [0, T]$ is the space of piecewise continuous and bounded on $[0, T]$ $n$-dimensional vector-functions. In the paper we will also require the space $L^2_n [0, T]$ of square-summable in $[0, T]$ $n$-dimensional vector-functions. Let $X$ be a normed space, then $||\cdot||_X$ denotes the norm in this space and $X^*$ denotes the space conjugate to the space $X$.

In the paper we assume that each trajectory $x(t)$ is a piecewise continuously differentiable vector-function with bounded derivative in its domain. Let $t_0~\in~[0, T)$ be a point of nondifferentiability of the vector-function $x(t)$, then for definiteness we assume that $\dot x(t_0)$ is a right-hand derivative of the vector-function $x(t)$ at the point~$t_0$. Similarly, we assume that $\dot x(T)$ is a left-hand derivative of the vector-function $x(t)$ at the point~$T$. With the assumptions and the notations made we can suppose that the vector function $x(t)$ belongs to the space $C_{n} [0, T]$ and that the vector function $\dot x(t)$ belongs to the space~$P_{n} [0, T]$.

For the arbitrary set $F \subset R^n$ let us define a support function of the vector $\psi~\in~R^n$ as $c(F, \psi) = \sup \limits_{f\in F}\langle f, \psi \rangle$ where $\langle a, b \rangle$ is a scalar product of the vectors $a, b \in R^n$. Let also $S_n$ be a unit sphere in $R^n$ with the center in the origin, let $B_r(c)$ be a ball with the radius $r \in R$ and the center $c \in R^n$, and the vectors $\bf{e_i}$, $i =\overline{1,n}$, form the standard basis in $R^n$. $0_n$ denotes a zero element of a functional space of some $n$-dimensional vector-functions, and $\bf{0_n}$ --- a zero element of the space $R^n$. Let $E_m$ denote an identity matrix and ${\bf {O}_{m}}$ --- a zero matrix in the space $R^m \times R^m$, let also $\mathrm{diag[P, Q]}$ denote a diagonal matrix with the blocks $P$ and $Q$ (where $P$ and $Q$ are the matrices of some dimensions). Denote $|x| = \sum_{i=1}^n |x_i|$, where $x \in R^n$.

In the paper we will use both superdifferentials of functions in a finite-dimensional space and superdifferentials of functionals in a functional space. Despite the fact that the second concept generalizes the first one, for convenience we separately introduce definitions for both of these cases and for those specific functions (functionals) and their variables and spaces which are considered in the paper. 
\begin{definition}
Consider the space $R^n \times R^n$ with the standard norm. Let \linebreak $d = [d_1, d_2] \in R^n \times R^n$ be an arbitrary vector. Suppose that at the point $(x, z)$ there exists such a convex compact set $\overline \partial h(x,z)$ $\subset R^n \times R^n$ that 
\begin{equation}
\label{0.1}
\frac{\partial h(x,z)}{\partial d} = \lim_{\alpha \downarrow 0} \frac{1}{\alpha} \big(h(x+\alpha d_1, z + \alpha d_2) - h(x,z)\big) = \min_{w \in \overline \partial h(x,z)} \langle w, d \rangle. 
\end{equation}

In this case the function $h(x,z)$ is called superdifferentiable at the point $(x, z)$, and the set $\overline \partial h(x,z)$ is called the superdifferential of the function $h(x,z)$ at the point $(x,z)$.

From expression (\ref{0.1}) one can see that the following formula
$$ h(x + \alpha d_1, z + \alpha d_2) = h(x, z) + \alpha  \frac{\partial h(x,z)}{\partial d} + o(\alpha, x, z, d), $$
$$\quad  \frac{o(\alpha, x, z, d)}{\alpha} \rightarrow 0, \ \alpha \downarrow 0,$$
holds true. 
\end{definition}

If for each number $\varepsilon > 0$ there exist such numbers $\delta > 0$ and $\alpha_0 > 0$ that at $\overline d \in B_{\delta}(d)$ and at $\alpha \in (0, \alpha_0)$ one has $| o(\alpha, x, z, \overline d) | < \alpha \varepsilon $, then the function $h(x,z)$  is called uniformly superdifferentiable at the point $(x, z)$. Note \cite{demvas} that if the function $h(x,z)$ is superdifferentiable at the point $(x, z)$ and is locally Lipschitz continuous in the vicinity of the point $(x, z)$, then it is uniformly superdifferentiable at the point $(x, z)$. 

If the function $\varsigma(\xi)$ is differentiable at the point $\xi_0 \in R^\ell$, then its superdifferential at this point is represented in the form \begin{equation}
\label{0.4}\overline \partial \varsigma(\xi_0) = \{ \varsigma'(\xi_0) \}, 
\end{equation}
 where $\varsigma'(\xi_0)$ is a gradient of the function $\varsigma(\xi)$ at the point $\xi_0$.
Note also that the superdifferential of the finite sum of superdifferentiable functions is the sum of the superdifferentials of summands, i. e. if the functions $\varsigma_k(\xi)$, $k = \overline{1, r}$, are superdifferentiable at the point $\xi_0 \in R^\ell$, then the function $\varsigma(\xi) = \sum_{k=1}^r \varsigma_k(\xi)$ superdifferential at this point is calculated by the formula 
\begin{equation}
\label{0.5} \overline \partial \varsigma(\xi_0) = \sum_{k=1}^r \overline\partial \varsigma_k(\xi_0).
\end{equation}

\begin{definition}
Consider the space $C_n[0, T]$ with the $L_2^n [0, T] $ norm. Let $g \in C_n[0, T]$ be an arbitrary vector-function. Suppose that at the point $x$ there exists a convex weakly* compact set $\overline \partial {I(x)} \subset \big( C_n[0, T], || \cdot ||_{L_2^n [0, T]} \big) ^*$ such that  
\begin{equation}
\label{3'} 
\frac{\partial I(x)}{\partial g} = \lim_{\alpha \downarrow 0} \frac{1}{\alpha} \big(I(x+\alpha g) - I(x)\big) = \min_{w \in \overline \partial I(x)} w(g). \end{equation}

In this case the functional $I(x)$ is called superdifferentiable at the point $x$, and the set $\overline \partial {I(x)}$ is called the superdifferential of the functional $I(x)$ at the point $x$.

From expression (\ref{3'}) one can see that the following formula
$$I(x + \alpha g) = I(x) + \alpha  \frac{\partial I(x)}{\partial g} + o(\alpha, x, g),$$ $$ \quad  \frac{o(\alpha, x, g)}{\alpha} \rightarrow 0, \ \alpha \downarrow 0,$$
holds true.
\end{definition}

\begin{remark}
Note the following fact. Since, as is known, the space \linebreak $\big( C_n[0, T], || \cdot ||_{L_2^n [0, T]} \big)$ is everywhere dense in the space $L_2^n [0, T]$, then the space $\big( C_n[0, T], || \cdot ||_{L_2^n [0, T]} \big)^*$ is isometrically isomorphic to the space $L_2^n [0, T] $ (see \cite{KolmFom}); therefore, henceforth, we will identify these spaces \bigg($\big( C_n[0, T], || \cdot ||_{L_2^n [0, T]} \big)^*$ and $L_2^n [0, T]$\bigg).
\end{remark}

\section{Statement of the problem}

Consider the system of differential equations
\begin{equation}
\label{1}
\dot{x} = A x + B u
\end{equation}
with the initial point
\begin{equation}
\label{2}
x(0) = x_{0}
\end{equation}
and with the desired endpoint
\begin{equation}
\label{3}
x_j(T) = {x_T}_j.
\end{equation}
In formula (\ref{1}) $A$ is a constant $n \times n$ matrix, $B$ is a constant $n \times m$ matrix. For simplicity we suppose that $B = \mathrm{diag} [E_m, {\bf{O}_{n-m}}]$. The system is considered on the given finite time interval $[-t^*, T]$ (here $T$ is a given final time moment; see comments on the time moment $t^*$ below). Then by section 2 assumption $x(t)$ is an $n$-dimensional continuous vector-function of phase coordinates with a piecewise-continuous and bounded on $[-t^*, T]$ derivative; the structure of the $m$-dimensional control $u$ will be specified below. In formula (\ref{2}) $x_0 \in R^n$ is a given vector; in formula (\ref{3}) ${x_T}_j$ are given numbers, corresponding to those coordinates of the state vector, which are fixed at the right endpoint, here $j \in J \subset \{1..n\}$, where $J$ is a given index set.  

Let also the discontinuity surface 
\begin{equation}
\label{4}
s(x, c) = {\bf 0_m}
\end{equation}
be given, where $s(x, c)$ is a continuously differentiable vector-function. We will assume that the general structure of the surface is given while the parameters $c \in R^\ell$ are unknown. From sliding modes control theory point of view it is natural to restrict ourselves to considering the hyperplanes $$s_j(x, c) = \displaystyle \sum_{i=1}^n c_{i,j} x_i + c_{n+1, j}, \quad j = \overline{1, m},$$ where some of the components of $\{c_{i,j}\}$, $i = \overline{1, n+1}$, $j = \overline{1, m}$, are to be determined. We will sometimes omit the dependence of the surface vector-function on these variables for convenience of notation. 

\begin{remark}
In practice, the general structure of the discontinuity surface is usually known based on some physical considerations; one also can  approximately know the parameters of this surface. So it makes sense to ``correct'' the parameters of this surface in such a way that the object is endowed with the desired properties (for example, the desired condition on the right endpoint).   
\end{remark}

 Let us immediately write out the explicit form of the control which is mainly used in this paper. Let
\begin{equation}
\label{5}
u_i = -\alpha_i |x| \mathrm{sign}(s_i(x)), 
\end{equation} 
where $i = \overline{1, m}$, $\alpha_i \in [\underline {a}_i, \overline {a}_i]$, $i = \overline{1, m}$, are some positive numbers which are sometimes called gain factors.

In book \cite{Utkin} it is shown that if surface (\ref{4}) is a hyperplane, then under natural assumptions and with sufficiently large values of the factors $\alpha_i$, $i = \overline{1,m}$, controls (\ref{5}) ensure system (\ref{1}) hitting a small vicinity of discontinuity surface (\ref{4}) from arbitrary initial state (\ref{2}) in the finite time $t^*$ and further staying in this neighborhood with the fulfillment of the condition   $s_i(x(t)) \rightarrow 0$, $i = \overline{1,m}$, at $t \rightarrow \infty$, i. e. controls (\ref{5}) ensure the stability ``in big'' of the system (\ref{1}) sliding mode. In \cite{Utkin} one may also find the estimates on the time moment $t^*$. Here we assume that all the conditions required are already met, i. e. the numbers $\underline a_i$, $\overline a_i$, $i = \overline{1, m}$, are taken sufficiently large. In sections~4-6 we will be interested in the behavior of the system on the discontinuity surface (on the time interval $[0,T]$). Section~7 at the end of the paper also discusses the search for the trajectory of the system both on the discontinuity surface and before hitting it; herewith, a different (from that given in the formula (\ref{5})) control structure will be used.
 


On the surfaces $s_i(x) = 0$, $i = \overline{1,m}$, the solution of the original discontinuous system (accordingly to the definition given in Introduction) is a solution of the following differential inclusion:
\begin{equation}
\label{8}
\dot x_i \in A_i x + [\underline a_i, \overline a_i] |x| [-1, 1] = A_i x + [-\overline a_i, \overline a_i] |x|, \quad i = \overline {1,m}.
\end{equation}
\begin{equation}
\label{9}
\dot x_i = A_i x, \quad i = \overline{m+1, n}.
\end{equation}

In formulas (\ref{8}), (\ref{9}), $A_i$ is the $i$-th row of the matrix $A$, $i = \overline{1,n}$. Note that the more detailed version of the definition of a discontinuous system solution used in this paper is given in \cite{AizerPyatn}. It has a strict and rather complicated form so we don't consider the details here. For our purposes it is sufficient to postulate that the system (\ref{1}) solutions with controls (\ref{5}) employed are the solutions of system (\ref{8}), (\ref{9}) by definition.

 \begin{remark}
 It is clear that in practice we believe that the trajectory is on the discontinuity surface if it lies in a given sufficiently small neighborhood of this surface (which can be determined in each specific case based on the requirements for the accuracy of calculations and modeling of the process described by the system). 
\end{remark}

Let $F(x) = (f_1(x), \dots, f_n(x))'$, where $f_1(x), \dots, f_n(x)$ run through the corresponding sets $F_1(x), \dots F_n(x)$ from the right-hand sides of inclusions (\ref{8}); herewith $f_i(x) = F_i(x) := A_i x $, $i = \overline{m+1,n}$ (see (\ref{9})), so we can rewrite the given inclusions in the form 
\begin{equation}
\label{2.10}
\dot{x} \in F(x).
\end{equation}

 We formulate the problem as follows: it is required to find such a trajectory $x^{*} \in C_{n}[0, T]$ (with the derivative $\dot x^{*}~\in~P_{n}[0, T]$) which moves along discontinuity surface (\ref{4}) (the parameters $c^* \in R^\ell$ are to be determined as well) while $t~\in [0,T]$, satisfies differential inclusion (\ref{2.10}) and boundary conditions (\ref{2}), (\ref{3}). It is apparent that due to the continuity requirement on the desired trajectory, to the presence of restrictions on the right endpoint and to the complicated discontinuous surface structure which is unknown in advance, the classical theorems on the existence of a solution are not applicable here. The existence problem of the system considered is a complicated one and is beyond the scope of this paper. One can only say that as it has been noted in Introduction if from engineering practice it is known that a solution exists for some ``initial'' values of the desired surface parameters $c$ and for some endpoint, then it is sensible to state a problem of ``correcting'' these parameters in such a way that the desired properties of the trajectory are achieved. So it is natural to assume that there exists a problem solution if such a ``correction'' is not too significant. These considerations justify (at a qualitative level) such a problem statement and the assumption that there exists a corresponding solution.


\begin{remark}
Instead of trajectories from the space $C_n [0, T]$ with derivatives from the space $P_n [0, T]$ one may consider absolutely continuous on the interval $[0, T]$ trajectories with measurable and almost everywhere bounded on $[0, T]$ derivatives, what is more natural for differential inclusions. The choice of the solution space in the paper is explained by the possibility of its practical construction.
\end{remark} 

\section{Reduction to a variational problem}
We will sometimes write $F$ instead of $F(x)$ for brevity. 
Insofar as  $\forall x \in R^n$ the set $F(x)$ 
is a convex compact set in $R^n$, then inclusion (\ref{2.10}) 
may be rewritten as follows {\cite{Blagodatskih}}:
$$
\dot x_i(t) \psi_i \leqslant c(F_i(x(t)), \psi_i) \quad \forall \psi_i \in S_1, \quad \forall t \in [0, T], \quad i = \overline{1,n}.
$$

Calculate the support function of the set $F_i$.  For this note that the set $F_i$ is a one-dimensional ``ball'' with the the center 
$$ c_i(x) = A_i x , \quad i = \overline{1,n},$$
and with the ``radius'' 
$$ r_i(x) = \overline a_i |x|, \quad i = \overline {1,m},$$
$$ r_i(x) = 0, \quad i = \overline{m+1, n}.$$

So the support function of the set $F_i$ can be expressed \cite{Blagodatskih} by the formula
$$ c(F_i(x),\psi_i) = \psi_i A_i x + \overline a_i |x| |\psi_i|, \quad i = \overline {1,m},$$
$$ c(F_i(x),\psi_i) = \psi_i A_i x, \quad i = \overline{m+1, n}.$$

%
 
We see that the support function of the set $F_i$ is continuously differentiable in the phase coordinates $x$ if $x_i \neq 0$, $i = \overline{1, n}$. 

Denote $z(t) = \dot x(t)$, $z \in P_n [0, T]$, then from (\ref{2}) one has
\begin{equation}
\label{3.11''}
x(t) = x_0 + \int_0^t z(\tau) d\tau.
\end{equation}

Put
\begin{equation}
\label{3.12}
\ell_i(\psi_i, x, z) = \langle z_i, \psi_i \rangle - c ( F_i(x), \psi_i ),
\end{equation}
$$
h_i(x, z) = \max_{\psi_i \in S_1} \max \{ 0, \ell_i(\psi_i, x, z) \},
$$
$$h(x,z) = (h_1(x,z), \dots, h_n(x,z))'$$

and construct the functional
\begin{equation}
\label{3.13}
\varphi(z) = \frac{1}{2} \int_0^T h^2 \Big( x_0 + \int_0^t z(\tau) d\tau, z(t) \Big) dt.
\end{equation}

Consider the set
 $$
 \Omega = \{z \in P_{n} [0, T] \ | \ \varphi(z) = 0 \}.
 $$
It is not difficult to check that for functional (\ref{3.13}) the relation
\begin{equation} \label{3.133}
\left\{
\begin{array}{ll}
\varphi(z) = 0 \ (z \in \Omega), \ &\text{if} \ \dot x_i(t) \psi_i  \leqslant c(F_i(x(t)), \psi_i) \quad \forall \psi_i \in S_1, \quad \forall t \in [0, T], \quad i = \overline{1,n}. \\
\varphi(z) > 0 \ (z \notin \Omega), \ &\text{otherwise},
\end{array}
\right.
\end{equation}
holds true, i. e. inclusion (\ref{2.10}) takes place iff $\varphi(z) = 0$.

Introduce the functional
\begin{equation}
\label{3.14}
\chi(z) = \frac{1}{2} \sum_{j \in J} \left( {x_0}_j + \int_0^T z_j(t) dt - {x_T}_j \right)^2.
\end{equation}

It is seen that condition (\ref{2}) on the left endpoint is automatically satisfied due to the vector-function $z(t)$ definition and condition (\ref{3}) on the right endpoint is satisfied iff $\chi(z) = 0$.

As it was noted above, the search for a solution of a differential inclusion is carried out on the discontinuity surfaces $s_i(x) = 0$, $i = \overline {1, m}$, so also introduce the functional 
\begin{equation}
\label{3.14'}
\omega(z, c) = \frac{1}{2} \int_0^T s^2 \Big( x_0 + \int_0^t z(\tau) d\tau, c \Big ) dt.
\end{equation}

Construct the functional
\begin{equation} 
\label{3.15} 
I(z, c) = \varphi(z) + \chi(z) + \omega(z, c).
\end{equation}

So the original problem has been reduced to minimizing functional (\ref{3.15}) on the space $P_n [0, T] \times R^{\ell}$. Denote $z^*, c^*$ a global minimizer of this functional. Then $$ x^*(t) = x_0 + \int_0^t z^*(\tau) d\tau $$
is a solution of the initial problem (and the vector $c^*$ defines the discontinuity surface structure). 

%

\begin{remark}
The structure of the functional $\varphi(z)$ is natural as the value \linebreak $ h_i(x(t), z(t)) $, $i = \overline{1, n}$,
at each fixed $t \in [0, T]$ is just the Euclidean distance from the point $z_i(t)$
to the set $F_i (x(t))$; functional (\ref{3.13}) is half the sum of squares of the deviations in $L^2_n[0,T]$ norm of the trajectories $z_i(t)$ from the sets $F_i(x)$, $i = \overline{1, n}$, respectively; the meaning of functionals (\ref{3.14}), (\ref{3.14'}) structures is obvious.
\end{remark}


\section{Necessary minimum conditions of the functional ${I(z, c)}$ \\ in a particular case}

It is obvious that the point $x^*$ is Problem 1 solution iff the functional $I(z, c)$ vanishes at the corresponding point, i. e. $I(z^*, c^*) = 0$. In order to obtain a more constructive minimum condition (which is useful for developing a numerical method for solving the original problem), let us study the differential properties of the functional $I(z, c)$. Suppose that the trajectories $x_i(t)$, $i = \overline{1,n}$, vanish only at isolated time moments of the interval $[0, T]$. This assumption is natural if the discontinuity surfaces do not contain any of these trajectories vanishing on the interval $[0, T]$ subset of nonzero measure. Then similarly to work \cite{Fominyh1}, under the assumption made, it is proved that the functional $I(z, c)$ is G${\rm\hat{a}}$teaux differentiable.

The proof is carried out with the help of classical variations of the functionals $\varphi(z)$, $\chi(z)$ and $\omega(z, c)$ and uses Lebesgue's dominated convergence theorem, as well as such known facts as the support function additivity in the first argument, Lagrange's mean value theorem and integration by parts. Let us formulate the theorem (see the definition of the function $\psi_i^*(x,z)$, $i = \overline{1, n}$, in the next paragraph).

\begin{theorem}
\label{th:1}
{Let the trajectories $x_i(t)$, $i = \overline{1,n}$, vanish only at isolated time moments of the interval $[0, T]$. Then the functional $I(z, c)$ is G${\rm\hat{a}}$teaux differentiable and its gradient at the point $(z, c)$ is expressed by the formula
$$
\nabla I(z, c) = \Bigg[ \sum_{i=1}^{n}{h_i(x, z)} \psi_i^{*}(x,z) {\bf e_i} -$$ $$- \sum_{i=1}^{n} \int_{t}^{T} {h_i(x(\tau), z(\tau))}
 \frac{\partial c_i(F_i(x(\tau)), \psi_i^{*}(x(\tau),z(\tau)))}{\partial x} d\tau  +
$$
$$
+ \sum_{j \in J} \left( {x_0}_j + \int_0^T z_j(t) dt - {x_T}_j \right) {\bf e_j} + \sum_{i=1}^m \int_t^T s_i(x(\tau), c) \frac{\partial s_i(x(\tau), c)} {\partial x} d \tau, $$ $$  \sum_{i=1}^m \int_0^T s_i(x(\tau), c) \frac{\partial s_i(x(\tau), c)} {\partial c} d \tau \Bigg].
$$
}
\end{theorem}
\begin{proof}
 The detailed proof is carried out similarly as in \cite{Fominyh1}.
 \end{proof}

 Note that for the functional $I(z, c)$ G${\rm\hat{a}}$teaux differentiability it is significantly that the vector $\psi_i^{*}(x, z)$ is unique in the case $ \ell_i (\psi_i, x, z) > 0$, $i \in \{1..n\}$; namely:
 due to the structure of functional (\ref{3.12}) it is easy to check that in the case $ \ell_i (\psi_i, x, z) > 0$
 maximum of the expression
$ \max\{0, \ell_i ({\psi_i}, x, z)\} =$ $= \ell_i ({\psi_i}, x, z) $
is achieved at the only element $\psi_i^{*}(x, z) \in S_1$, $i \in \{1..n\}$.
A simple justification of this fact is carried out as in \cite{DolgFom} and is based on the convexity of the set $F_i(x)$, $i = \overline{1,n}$, at each $x$ and on the known properties of a support function. In the case $\ell_i (\psi_i, x, z) \leq 0$ we have put $\psi_i^{*}(x,z) = \psi_0$, $i \in \{1..n\}$, where $\psi_0 \in S_1$ is fixed. These considerations are valid for a more general case considered in work \cite{Fominyh1}. Note that in the particular case of this paper one can write out the explicit expression for the  vector $\psi_i^{*}(x, z)$ in the case $ \ell_i (\psi_i, x, z) > 0$, $i \in \{1..n\}$, as follows: 
$$\psi_i^{*}(x, z) = \mathrm{sign} \left( z_i - A_i x \right),$$
note that $ z_i - A_i x  \neq 0$ in the case $ \ell_i (\psi_i, x, z) > 0$.

The following theorem formulates the known minimum condition for a Gateaux differentiable functional. 

\begin{theorem}
\label{th:2}
{Let the trajectories $x_i(t)$, $i = \overline{1,n}$, vanish only at isolated time moments of the interval $[0, T]$. In order for the point $(z^*, c^*)$ to minimize the functional $I(z, c)$, it is necessary that
$$
0_n \times  {\bf 0_\ell} = \Bigg[ \sum_{i=1}^{n}{h_i(x^*, z^*)} \psi_i^{*}(x^*,z^*) {\bf e_i} -$$ $$- \sum_{i=1}^{n} \int_{t}^{T} {h_i(x^*(\tau), z^*(\tau))}
 \frac{\partial c_i(F_i(x^*(\tau)), \psi_i^{*}(x^*(\tau),z^*(\tau)))}{\partial x} d\tau  +
$$
$$
+ \sum_{j \in J} \left( {x_0}_j + \int_0^T z^*_j(t) dt - {x_T}_j \right) {\bf e_j} + \sum_{i=1}^m \int_t^T s_i(x^*(\tau), c^*) \frac{\partial s_i(x^*(\tau), c^*)} {\partial x} d \tau, $$ $$ \sum_{i=1}^m \int_0^T s_i(x^*(\tau), c^*) \frac{\partial s_i(x^*(\tau), c^*)} {\partial c} d \tau \Bigg]
$$
where $0_n$ is a zero element of the space $P_n[0, T].$
}
\end{theorem}

It is apparent, that $I(z^*, c^*) = 0$ is necessary and sufficient minimum condition for this functional and in this case the equality of Theorem 2 is automatically satisfied.

\begin{remark}
Note that formally the behavior of an object is described by differential inclusion (\ref{2.10}) only if $s(x, c) = 0$. In practice it is obvious that applying any optimization method, one is able to find only an approximate minimizer of functional $I(z, c)$, so if the approximate global minimizer $(\overline x, \overline c)$ is obtained, then one has $ \omega(\overline x, \overline c) \leq \varepsilon_s$. On the one hand, one can be satisfied with this (set in advance) accuracy based on physical considerations (see Remark 3). On the other hand, in some problems it is possible to provide the exact equality $ s(\overline x, \overline c) = 0$ (for example, by expressing one of the variables through the obtained ones) and then checking that such a substitution does not affect the fulfillment of differential inclusion and restriction on the right endpoint (see Example 1 where such a procedure is implemented). In practice the restriction on the right endpoint is also satisfied with some permissible accuracy, so one has $\chi(z) \leq \varepsilon_{x_T}$. 

Also note that the desired trajectory functional space is not closed in $L_n^2[0, T]$ metric so formally some optimization method can lead to unacceptable points at some iterations. But since the functional space considered is everywhere dense in the space $L_n^2[0, T]$ in practice we just approximate such a point with an acceptable one during numerical procedure.
\end{remark}

{\bf Example 1.}
Consider the following system
$$\dot x_1(t) \in [-1, 1] (|x_1(t)|+|x_2(t)|+|x_3(t)|), \quad \dot x_2(t) = x_1(t) - x_2(t), \quad \dot x_3(t) = x_2(t)$$
with the boundary conditions
$$x_1(0) = -3, \ x_2(0) = 4, \ x_3(0) = 6, \quad x_1(1) = 0.$$
Let the discontinuity surface in this example be of the form
$$s(x, c) = x_1 + c_1 x_2 + c_2 = 0.$$

A simplest steepest descent method (in the functional space) \cite{Kantorovich} with some modifications (see Remark 7, b) below) was used in order to minimize the functional $I(z, c)$ in this problem. 

The point $(z_{\{0\}}, c_{\{0\}}) = (0, 0, 0, 1.25, -1.25)'$ was taken as the initial one (note that here $(0, 0, 0)'$ is a zero point in the functional space $P_3 [0,1]$ and the point $(1.25, -1.25)'$ belongs to the space $R^2$). Herewith, we have $I(z_{\{0\}}, c_{\{0\}}) \approx 37.28125$. Put \linebreak $x_1(t) := -c_{\{0\}1} x_2(t) - c_{\{0\}2}$ and integrate the last two equations of the system given with one of the known numerical methods (for this example the Runge-Kutta 4-5-th order method was used). Check that the first inclusion is also satisfied and finally have $x_1(1) \approx 0.10175$; we see that the error on the right endpoint is of the order $10^{-1}$. 

At the 60-th iteration the point $(z_{\{60\}}, c_{\{60\}})' $ was constructed and we approximately put $(z^*, c^*) = (z_{\{60\}}, c_{\{60\}})$. Herewith $c^* \approx (0.98467, -0.93868)'$, $I(z^*, c^*) \approx 0.00015$ and $|| \nabla I(z^*, c^*) ||_{L^2_3 [0,1] \times R^2} \approx 0.02009$. The point $z^*$ is not given here for two reasons: 1) for brevity (as it is constructed in the form of a rather bulky piecewise continuous vector-function) and 2) because only the parameters $c^*$ are finally used to estimate the result obtained (see the next paragraph). 

Having obtained the parameters $c^*$, put $x_1(t) := -c_1^* x_2(t) - c_2^*$ and substitute them into the last two equations of the system given. Integrate this closed-loop system with the Runge-Kutta 4-5-th order method, then check that the first inclusion is also satisfied and finally have $x_1(1) \approx -0.00431$, so we see that the desired value on the right endpoint is achieved with an error of the order $5 \times 10^{-3}$. So this value has been improved via ``correcting'' the parameters of the surface considered (see Remark 2). Picture 1 demonstrates the results of calculations. The black lines denote the curves which were obtained via the method of the paper, while the dashed red lines denote the curves obtained via integration of the closed-loop system (and via the relation $x_1(t) = -c_1^* x_2(t) - c_2^*$).

   \begin{figure*}[h!]
\begin{minipage}[h]{0.3\linewidth}
\center{\includegraphics[width=1\linewidth]{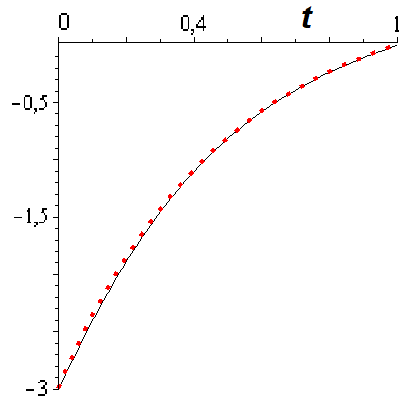} }
\end{minipage}
\hfill
\begin{minipage}[h]{0.3\linewidth}
\center{\includegraphics[width=1\linewidth]{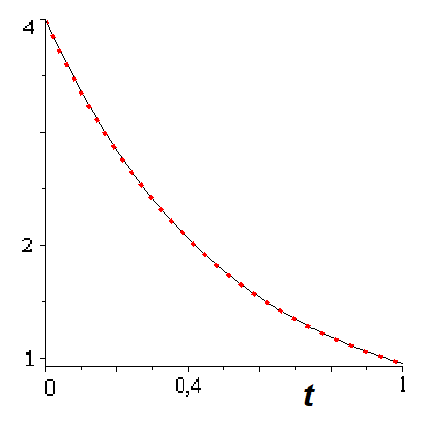} }
\end{minipage}
\hfill
\begin{minipage}[h]{0.3\linewidth}
\center{\includegraphics[width=1\linewidth]{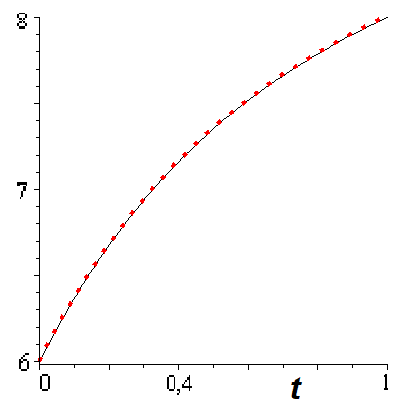} }
\end{minipage}
\caption{Example 1. The trajectories $x_1^*(t)$, $x_2^*(t)$, $x_3^*(t)$.}
\end{figure*}

\begin{remark}
a) In practice differential inclusions (\ref{8}) are satisfied at most iterations, since with $i = \overline{1, m}$ the function $z_i(t) - A_i  x(t)$ values are usually not very large, so they are inside the interval $[-\overline{a}_i, \overline{a}_i] |x(t)|$ during the interval $[0, T]$ (recall that the gain factor $\overline{a}_i$ is sufficiently large). This fact significantly simplifies calculations insofar as in the case $\varphi(z) > 0$ (see (\ref{3.133})) the functional~$\varphi(z)$ gradient is the most complicated (comparatively to the other functionals considered).  
  
b) During computation procedure it was noted that $z_i$, $i = \overline{1, n}$, are so-called ``fast'' variables while $c_k$, $k = \overline {1, \ell}$, are so called ``slow'' variables. This means that the trajectories may change relatively significantly with the fixed parameters while the variables~$c_k$, $k = \overline {1, \ell}$, change slightly with the fixed variables $z_i$, $i = \overline{1, n}$,  throughout iterations. So in practice it was reasonable to fix the parameters $c_k$, $k = \overline {1, \ell}$, and to make one iteration with a relatively large descent step (in order to minimize the functional in the variables $z_i$, $i = \overline{1, n}$) and then to fix the trajectories $z_i$, $i = \overline{1, n}$, and to make several iterations with small descent steps (in order to minimize the functional in the variables $c_k$, $k = \overline{1, \ell}$) and to repeat this process.
\end{remark}

\section{Differential properties of the functional $I(x, c)$ \\ in a more general case}
In this section, it will be more convenient for presentation not to go to the space of derivatives but to work with the following functional (we keep the previous notation for it): 
$$I(x, c) = \varphi(x) + \chi(x) + \omega(x, c),$$
where as previously (we omit the $1/2-$factor in the functional $\varphi(x)$ for its more convenient exploration below) 
\begin{equation}
\label{6.1}
\varphi(x) =  \int_0^T h^2 \big( x(t), \dot x (t) \big) dt, 
\end{equation}
$$\chi(x) = \frac{1}{2} \sum_{j \in J} \left( {x_0}_j + \int_0^T \dot x_j(t) dt - {x_T}_j \right)^2,$$
$$\omega(x, c) = \frac{1}{2} \int_0^T s^2 (x(t), c)  dt.$$

In section 5 we have shown that the functional $I(z, c)$ is G${\rm\hat{a}}$teaux differentiable, and according to the assumption made there, the case was excluded when at least one of the phase variables would be identically equal to zero on some time interval of the set $[0, T]$ of nonzero length. Let us now assume that some of the trajectories can be identically equal to zero on some nonzero time interval of the segment $[0, T]$ and study the differential properties of the functional $I(z, c)$ in this case. 

For simplicity consider the case $n=2$ and only the functions $\ell_1(\psi_1, x, z)$ and $h_1(x, z)$  (here we denote them $\ell(\psi_1, x_1, x_2, z_1)$ and $h(x_1, x_2, z_1)$ respectively) and the time interval $[t_1, t_2] \subset [0, T]$ of nonzero length; the general case is considered in a similar way. Then we have $ \ell (\psi_1, x_1, x_2, z_1) = $ \linebreak $=z_1 \psi_1 -  a_1 x_1 \psi_1 - a_2 x_2 \psi_1 - b |x_1| |\psi_1| - b |x_2| |\psi_1|$, where $a_1 := a_{1,1}$, $a_2 := a_{1,2}$, $b := \overline a_1$. Fix some point $(x_1, x_2) \in R^2$. Let $x_1(t) = 0$, $z_1(t) = 0$, $x_2(t) > 0$ at $t \in [t_1, t_2]$; other cases are studied in a completely analogous fashion.

a) Suppose that $h_1 (x, z) > 0$, i. e. $h_1(x, z) = \max_{\psi_1 \in S_1}\ell_1 (\psi_1, x, z) > 0$. 

Our aim is to apply the corresponding theorem on a directional differentiability from \cite{BonnansShapiro}. The theorem of this book considers the inf-functions are considered so we will apply this theorem to the function $-\ell(\psi_1, x_1, x_2, z_1)$. For this check that the function $h(x_1, x_2, z_1)$ satisfies the following conditions: \newline
i) the function $\ell (\psi_1, x_1, x_2, z_1)$ is continuous on $S_1 \times R^2 \times R$; \newline
ii)  there exist a number $\beta$ and a compact set $C \in R$ such that for every $(x_1, x_2, z_1)$ in the vicinity of the point $(0, x_2, 0)$ the level set 
$$\mathrm {lev}_\beta \ell(\cdot, x_1, x_2, z_1) = \{\psi_1 \in S_1 \ | \ -\ell(\psi_1, x_1, x_2, z_1) \leqslant \beta\} $$
is nonempty and is contained in the set $C$; \newline
iii) for any fixed $\psi_1 \in S_1$ the function $\ell(\psi_1, \cdot, \cdot,\cdot)$ is directionally differentiable at the point $(0, x_2, 0)$; \newline
iv) if $d = [d_1, d_2] \in R^2 \times R$, $\gamma_n \downarrow 0$ and $\psi_{1_{n}}$ is a sequence in $C$, then $\psi_{1_{n}}$ has a limit point $\overline \psi_1$ such that 
$$\mathrm{lim} \sup\limits_{n \rightarrow \infty} \frac{-\ell(\psi_{1_{n}}, 0 + \gamma_n d_{1,1}, x_2+\gamma_n d_{1,2}, 0+\gamma_n d_2) - (-\ell(\psi_{1_{n}}, 0, x_2, 0))}{\gamma_n} \geqslant$$ $$\geqslant \frac{\partial(-\ell(\overline \psi_1, 0, x_2, 0))}{\partial d}, $$
where $\displaystyle{\frac{\partial\ell(\overline \psi_1, 0, x_2, 0)}{\partial d}}$ is the derivative of the function $\ell (\overline \psi_1, x_1, x_2, z_1)$ at the point $(0,x_2,0)$ in the direction $d$. 

The verification of conditions i), ii) is obvious. 

In order to verify condition iii), it suffices to note that since $b > 0$, then at the fixed $\psi_1 \in S_1$ the function $- b |x_1| |\psi_1|$ is superdifferentiable (and hence is differentiable in directions) at the point $0$, herewith, its superdifferential at the point $0$ is the segment $\mathrm{co}\{-b |\psi_1|, b |\psi_1|\}$. An explicit expression for the derivative of this function at the point~$0$ in the direction $d_{1,1}$ is $-b |d_{1,1}| |\psi_1|$. 

Finally, check condition iv). Let $[d_1, d_2] \in R^2 \times R$, $\gamma_n \downarrow 0$ and $\psi_{1_n}$ is some sequence from $C$. Calculate
$$\mathrm{lim} \sup\limits_{n \rightarrow \infty} \frac{\ell(\psi_{1_n}, 0 + \gamma_n d_{1,1}, x_2 + \gamma_n d_{1,2}, 0 + \gamma_n d_2) - \ell(\psi_{1_n}, 0, x_2, 0)}{\gamma_n} =$$ $$=\mathrm{lim} \sup\limits_{n \rightarrow \infty} \frac{\gamma_n d_2 \psi_{1_n} - \gamma_n a_1 d_{1,1} \psi_{1_n} - \gamma_n a_2 d_{1,2} \psi_{1_n} - b |\gamma_n d_{1,1}| |\psi_{1_n}|-b \gamma_n d_{1,2} |\psi_{1_n}|}{\gamma_n} =$$ $$ =\mathrm{lim} \sup\limits_{n \rightarrow \infty} \left( d_2 \psi_{1_n} - a_1 d_{1,1} \psi_{1_n} - a_2 d_{1,2} \psi_{1_n} - b |d_{1,1}| |\psi_{1_n}| - b d_{1,2} |\psi_{1_n}|  \right)$$
Let $\overline \psi_1$ be a limit point of the sequence $\psi_{1_n}$. Then by the directional derivative definition we have 
$$\frac{\partial\ell(\overline \psi_1, 0, x_2, 0)}{\partial d} = d_2 \overline \psi_1 - a_1 d_{1,1} \overline \psi_1 - a_2 d_{1,2} \overline \psi_1 - b |d_{1,1}| |\overline \psi_1| - b d_{1,2} |\overline \psi_1|.$$ From last two equalities one obtains that condition iv) is fulfilled .

Thus, the function $h(x_1, x_2, z_1)$ satisfies conditions i)-iv), so it is differentiable in directions at the point $(0, x_2, 0)$ \cite{BonnansShapiro}, and its derivative in the direction $d$ at this point is expressed by the formula
$$
\frac{\partial h(0, x_2, 0)}{\partial d} = \sup_{\psi_1 \in S(0, x_2, 0)} \frac{\partial \ell(\psi_1, 0, x_2, 0)}{\partial d},
$$
where $S(0, x_2, 0) = \mathrm{arg} \max_{\psi_1 \in S_1} \ell(\psi_1, 0, x_2, 0)$. However, as shown above, in the considered problem the set $S(0, x_2, 0)$ consists of the only element $\psi_1^*(0, x_2, 0)$, hence
$$
\frac{\partial h(0, x_2, 0)}{\partial d} = \frac{\partial \ell(\psi_1^*(0, x_2, 0), 0, x_2, 0)}{\partial d},
$$

Finally, recall that by the directional derivative definition one has the equality
$$\frac{\partial \ell(\psi_1^*(0, x_2, 0), 0, x_2, 0)}{\partial d} = d_2 \psi_1^* - a_1 d_{1,1} \psi_1^* - a_2 d_{1,2} \psi_1^* - b |d_{1,1}| |\psi_1^*| - b d_{1,2} |\psi_1^*|,$$
where we have put $\psi^*_1 := \psi_1^*(0, x_2, 0)$.

From last two equalities we finally obtain that the function $h(x_1, x_2, z_1)$ is superdifferentiable at the point $(0, x_2, 0)$, but it is also positive in this case, hence the function $h^2(x_1, x_2, z_1)$ is superdifferentiable at the point $(0, x_2, 0)$ as a square of a superdifferentiable positive function (see \cite{demrub}).

b) In the case $h_1 (x, z) = 0$ it is obvious that the function $h^2(x_1, x_2, z_1)$ is differentiable at the point $(0, x_2, 0)$ and its gradient vanishes at this point.

\begin{remark}
It is easy to see that the proof above may be significantly simplified if one uses the fact that $\psi_1, \psi_{1_n}, \overline \psi_1, \psi_1^* \in S_1$.  However, since the statement about the superdifferentiability of functions having a structure similar to that of the function $h_1(x, z)$ is of independent interest, so such a way of proof is used, which may be applied to the more general case when $\psi_1$ belongs to an arbitrary compact subset of the space $R$ (therefore, we consciously did not put $|\psi_1| = |\psi_{1_n}| = |\overline \psi_1| = |\psi_1^*| = 1$ in the proof). 
\end{remark}

Above we have considered a particular case. Dividing the interval $[0,T]$ into the segments in which some of the phase trajectories are identically equal to zero and some retain a definite sign (in this section we suppose that there is a finite number of such segments) and arguing in each of these intervals similarly to the case considered and also using the superdifferential calculus rules (\ref{0.4}), (\ref{0.5}), write down the superdifferential of the function $h_i(x,z)$. With $i = \overline{1,m}$ one has
$$\overline \partial h_i(x,z) = \psi^*_i {\bf e_{i+n}} - \psi^*_i [A'_i, {\bf 0_n}] + \sum_{j=1}^n \overline \partial(-\overline{a}_i |x_j| |\psi_i^*|),$$
where with $j = \overline{1,n}$ we have
$$ \overline \partial(-\overline{a}_i |x_j| |\psi_i^*|) = \left\{
\begin{array}{lll}
-\overline{a}_i |\psi_i^*| {\bf e_j}, \ &\text{if} \ x_j > 0,  \\
\overline{a}_i  |\psi_i^*| {\bf e_j}, \ &\text{if} \ x_j < 0, \\
\mathrm{co} \big \{-\overline{a}_i |\psi_i^*| {\bf e_j},  \overline{a}_i |\psi_i^*| {\bf e_j} \big \}, \ &\text{if} \ x_j = 0.
\end{array}
\right.
$$
With $i = \overline{m+1,n}$ one has
$$\overline \partial h_i(x,z) = \psi^*_i {\bf e_{i+n}} - \psi^*_i [A'_i, {\bf 0_n}].$$

Thus, in a more general case considered in this section, when the segment~$[0, T]$ may be divided into a finite number of intervals, in every of which each phase trajectory is either identically equal to zero or retains a certain sign, the integrand of functional (\ref{6.1}) is superdifferentiable as a square of a superdifferentiable nonnegative function \cite{demrub}. It turns out that this fact allows us to conclude that functional (\ref{6.1}) itself is superdifferentiable (in the case considered). 
 
\begin{theorem}
\label{th:3}  
Let the interval $[0, T]$ may be divided into a finite number of intervals, in every of which each phase trajectory is either identically equal to zero or retains a certain sign. Then the functional
 $
 \varphi(x) 
 $
is superdifferentiable, i. e. 
\begin{equation}
\label{10} 
\frac{\partial \varphi(x)}{\partial g} = \lim_{\alpha \downarrow 0} \frac{1}{\alpha} \big(\varphi(x+\alpha g) - \varphi(x)\big) = \min_{w \in \overline \partial \varphi(x)} w(g) \end{equation}
Here the set $\overline \partial \varphi(x)$ is of the following form
\begin{equation}
\label{13} 
\overline \partial \varphi(x) = \Bigg\{ w \in \left(C_n[0, T], || \cdot ||_{L_2^n[0, T]} \right)^* \ \big|\end{equation}

$$ \ w(g) = \int_0^T \langle w_1(t), g(t) \rangle dt + \int_0^T \langle w_2(t), \dot g(t) \rangle dt \quad \forall g \in C_n[0,T], \, \dot g \in P_n[0,T],  
$$
$$
w_1(t), w_2(t) \in L^n_\infty[0, T], \quad [w_1(t), w_2(t)] \in \overline \partial h^2(x(t), \dot x(t)) \quad for \ a. e. \ t \in [0, T] \Bigg\}.$$
 \end{theorem}
\begin{proof}
In accordance with definition (\ref{3'}) of a superdifferentiable functional, in order to prove the theorem, one has to check that

1) the derivative of the functional $\varphi(x)$ in the direction $g$ is actually of form (\ref{10}), 

2) herewith, the set $\overline \partial \varphi(x)$ is convex and weakly* compact subset of the space $\left( C_n[0, T], || \cdot ||_{L_2^n [0, T]} \right)^*$. 

Let us prove statement 1). 

At first show that the following relation is true.

\begin{equation}
\label{16} 
\lim_{\alpha \downarrow 0}\frac{1}{\alpha} \Big | \varphi(x+\alpha g) - \varphi(x) - \int_0^T  \min_{[w_1, w_2] \in \overline \partial h^2(x,z)} \big( \langle w_1(t), \alpha g(t) \rangle + \langle w_2(t), \alpha \dot g(t) \rangle \big) dt \Big| = 0.
 \end{equation} 
 
 Denote 
 \begin{equation}
\label{16'} 
f(t, \alpha) = \frac{1}{\alpha} \Big ( h^2(x(t)+\alpha g(t), \dot x(t)+\alpha \dot g(t)) - h^2(x(t), \dot x(t)) \Big) - 
 \end{equation} 
 $$ - \min_{[w_1, w_2] \in \overline \partial h^2(x,z)} \big( \langle w_1(t), g(t) \rangle + \langle w_2(t), \dot g(t) \rangle \big) .$$

Our aim is to prove relation (\ref{16}) via Lebesgue's dominated convergence theorem applied to the function $f(t, \alpha)$ (at $\alpha \downarrow 0$).

At first note that by superdifferential definition (\ref{0.1}) and by the superdifferentiability of the function $h^2(x,z)$ (proved at the beginning of this section) for each $t \in [0, T]$ we have $f(t, \alpha) \rightarrow 0$ when $\alpha \downarrow 0$.  

In the following two paragraphs we show that for every $\alpha > 0$ one has \linebreak $f(t, \alpha) \in L^1_{\infty}[0,T]$.

Insofar as $x, g \in C_n[0,T]$, $z, \dot g \in P_n[0,T]$ and the function $h^2(x,z)$ is continuous in its variables due to its structure \cite{demmal}, we obtain that for each $\alpha > 0$ the functions $t \rightarrow h^2(x(t),z(t))$ and $t \rightarrow h^2(x(t)+\alpha g(t), z(t)+ \alpha \dot g(t))$ belong to the space $L_\infty^1 [0,T]$.

Due to the upper semicontinuity of a superdifferential mapping \cite {demvas} and the structure of the superdifferential $\overline \partial h^2(x, z)$, it is easy to check that the mapping $t \rightarrow \overline \partial h^2(x(t), z(t))$ is upper semicontinuous, so it is measurable (see \cite{filblag}). Then due to the continuity of the function $g(t)$, the piecewise continuity of the function $\dot g(t)$ and due to the continuity of the scalar product in its variables we obtain that for each $\alpha > 0$ the mapping 
\begin{equation}
\label{6.18} t \rightarrow \min_{[w_1, w_2] \in \overline \partial h^2(x(t),z(t))} \big( \langle w_1(t), \alpha g(t) \rangle + \langle w_2(t), \alpha \dot g(t) \rangle \big) \end{equation} is upper semicontinuous \cite{aubenfr}, and then is also measurable \cite{filblag}. While proving statement 2) it will be shown that under the assumptions made the set $\overline \partial h^2(x, z)$ is bounded uniformly in $t \in [0, T]$, then by the continuity of the function $g(t)$ and the piecewise continuity of the function $\dot g(t)$ it is easy to check that for each $\alpha > 0$ mapping (\ref{6.18}) is also bounded uniformly in $t \in [0, T]$. So we finally have that for each $\alpha > 0$ mapping (\ref{6.18}) belongs to the \linebreak space $L_\infty^1 [0,T]$. 

Now we prove that the function {$f(t, \alpha)$} is dominated by some integrable function for all sufficiently small $\alpha > 0$. On the segments where the phase trajectories retain their signs this is done via standard technique using Lagrange's mean value theorem. Consider the interval $[t_1, t_2] \subset [0, T]$ of nonzero length. For simplicity consider the case $n=2$ and only the functions $\ell_1(\psi_1, x, z)$ and $h_1(x, z)$  (here denote them $\ell(\psi_1, x_1, x_2, z_1)$ and $h(x_1, x_2, z_1)$ respectively); the general case is considered in a similar way. Suppose that $x_2(t) = 0$  and $x_1(t) > 0$ at $t \in [t_1, t_2]$; other cases are studied in a completely analogous fashion.  In the case $h_1(x,z) = 0$ the proof is elementary.  So suppose $h_1 (x, z) > 0$, i. e. $h_1(x, z) = \max_{\psi_1 \in S_1}\ell (\psi_1, x, z) = \ell (\psi_1^*(x, z), x, z) $.  As is shown in the previous paragraph, the second summand in (\ref{16'}) is integrable functions. So it remains to consider for sufficiently small $\alpha > 0$ the first summand in (\ref{16'}), i.e. (in the considered case) the function 
 $$ \frac{1}{\alpha} \Big ( h_1^2(x +\alpha g , \dot x +\alpha \dot g ) - h_1^2(x , \dot x) \Big) = $$
 $$=\frac{1}{\alpha} \bigg( \Big[ (z_1 + \alpha \dot g_1) \psi_1^*(\alpha) - a_1(x_1 + \alpha g_1) \psi_1^*(\alpha) - a_2 \alpha g_2 \psi_1^*(\alpha) -$$ \begin{equation}
\label{6.18'}- b (x_1 + \alpha g_1) |\psi_1^*(\alpha)| - b |\alpha g_2| |\psi_1^*(\alpha)| \Big]^2 -\Big[ z_1 \psi_1^* - a_1 x_1 \psi_1^* - b x_1 |\psi_1^*|  \Big]^2 \bigg), \end{equation}
 where we have put $a_1 := a_{1,1}$, $a_2 := a_{1,2}$, $b :=\overline a_1$, $x_1 := x_1(t)$, $x_2 := x_2(t)$, $z_1 := z_1(t)$, $g := g(t)$, $\dot g := \dot g(t)$, $\psi_1^* := \psi_1^*(x, z)$, $\psi_1^*(\alpha) := \psi_1^*(x + \alpha g, z + \alpha \dot g)$ for brevity.
 
Consider only the summands 
$$\frac{z_1^2 (\psi_1^*(\alpha))^2 - z_1^2(\psi_1^*)^2}{\alpha} + \frac{a_1^2 x_1^2 (\psi_1^*(\alpha))^2 - a_1^2 x_1^2 (\psi_1^*)^2}{\alpha} +$$$$+ \frac{b^2 x_1^2 (\psi_1^*(\alpha))^2 - b^2 x_1^2 (\psi_1^*)^2}{\alpha} + \frac{-2 a_1 z_1 x_1 (\psi_1^*(\alpha))^2 + 2 a_1 z_1 x_1 (\psi_1^*)^2}{\alpha} +$$ $$+ \frac{-2 b z_1 |x_1| |\psi_1^*(\alpha)|\psi_1^*(\alpha)  + 2 b z_1 |x_1| |\psi_1^*| \psi_1^*}{\alpha} + \frac{2 a_1 b x_1 |\psi_1^*(\alpha)|\psi_1^*(\alpha) - 2 a_1 b x_1 |\psi_1^*|\psi_1^*}{\alpha}$$
of the right-hand side of equation (\ref{6.18'}). The first four summands in this expression are equal to zero for all $\alpha > 0$ as $|\psi_1^*(\alpha)| = |\psi_1^*| = 1$. The second two summands in this expression are equal to zero for all sufficiently small $\alpha > 0$ as the function $\psi_1^*(x, z)$ is continuous in $(x, z)$ and also as $|\psi_1^*(\alpha)| = |\psi_1^*| = 1$.

By direct calculation one can easily check that due to the continuity of the functions $x_1(t)$, $x_2(t)$ and due to the piecewise continuity of the functions $z_1(t)$ and $\psi_1^*(x(t), z(t))$, $t \in [t_1, t_2]$, other summands of the right-hand side of equation (\ref{6.18'})  are dominated by a piecewise continuous function for all sufficiently small $\alpha >0.$




Consider the functional $\displaystyle{\int_0^T  \min_{[w_1, w_2] \in \overline \partial h^2(x,z)} \big( \langle w_1(t), \alpha g(t) \rangle + \langle w_2(t), \alpha \dot g(t) \rangle \big) dt}$ in details. For each $\alpha > 0$ and for a. e. $t \in [0, T]$ we have the obvious inequality
$$
 \min_{[w_1, w_2] \in \overline \partial h^2(x,z)} \big( \langle w_1(t), \alpha g(t) \rangle + \langle w_2(t), \alpha \dot g(t) \rangle \big) \leqslant \langle w_1(t), \alpha g(t) \rangle + \langle w_2(t), \alpha \dot g(t) \rangle,
$$
where $[w_1(t), w_2(t)]$ is some measurable selector of the mapping $t \rightarrow \overline \partial h^2(x(t), \dot x(t))$ (due to the noted boundedness uniformly in $t \in [0, T]$ of the set $\overline \partial h^2(x, \dot x)$ we have $w_1, w_2 \in L_\infty^n [0,T]$), then taking into account the form of formula (\ref{13}) for each $\alpha > 0$ one also has the inequality 
$$
\int_0^T \min_{[w_1, w_2] \in \overline \partial h^2(x,z)} \big( \langle w_1(t), \alpha g(t) \rangle + \langle w_2(t), \alpha \dot g(t) \rangle \big) dt \leqslant $$ $$\leqslant \min_{w \in \overline \partial \varphi(x)} \int_0^T \langle w_1(t), \alpha g(t) \rangle + \langle w_2(t), \alpha \dot g(t) \rangle dt. 
$$
Insofar as for each $\alpha > 0$ and for a.e. $t \in [0, T]$ we have 
$$
\min_{[w_1, w_2] \in \overline \partial h^2(x,z)} \big( \langle w_1(t), \alpha g(t) \rangle + \langle w_2(t), \alpha \dot g(t) \rangle \big) \in$$
 $$\in \Big\{ \langle w_1(t), \alpha g(t) \rangle + \langle w_2(t), \alpha \dot g(t) \rangle \ \big| \ [w_1(t), w_2(t)] \in \overline \partial h^2(x(t), \dot x(t)) \Big\},
$$
and the set $\overline \partial h^2(x,z)$ is closed and bounded at each fixed $t$ by the superdifferential definition and the mapping $t \rightarrow \overline \partial h^2(x(t),z(t))$, as noted above, is upper semicontinuous and besides, the scalar product is continuous in its arguments and $g \in C_n[0,T]$, $\dot g \in P_n[0,T]$, then due to Filippov lemma \cite{Filippov2} there exists such measurable selector $[\overline{w}_1(t), \overline{w}_2(t)]$ of the mapping $t \rightarrow \overline \partial h^2(x(t),z(t))$ that for each $\alpha > 0$ and for a. e. $t \in [0, T]$ we have
$$
 \min_{[w_1, w_2] \in \overline \partial h^2(x,z)} \big( \langle w_1(t), \alpha g(t) \rangle + \langle w_2(t), \alpha \dot g(t) \rangle \big) = \langle w_1(t), \alpha g(t) \rangle + \langle w_2(t), \alpha \dot g(t) \rangle,
$$
so we have found the element $\overline{w}$ of the set $\overline \partial \varphi(x)$ which brings the equality in the previous inequality. Thus, one finally has
\begin{equation}
\label{14} 
\int_0^T \min_{[w_1, w_2] \in \overline \partial h^2(x,z)} \big( \langle w_1(t), \alpha g(t) \rangle + \langle w_2(t), \alpha \dot g(t) \rangle \big) dt = $$ $$= \min_{w \in \overline \partial \varphi(x)} \int_0^T \langle w_1(t), \alpha g(t) \rangle + \langle w_2(t), \alpha \dot g(t) \rangle dt. 
\end{equation}


From relations (\ref{16}), (\ref{14}) one obtains expression (\ref{10}). 

Prove statement 2). 

The convexity of the set $\overline \partial \varphi(x)$ follows directly from the convexity of the set $\overline \partial h^2(x,z)$ at each fixed $t \in [0, T]$.

Prove the boundedness of the set $\overline \partial h^2(x,z)$ uniformly in $t \in [0, T]$. Due to the upper semicontinuity of the mapping $t \rightarrow \overline \partial h^2(x(t),z(t))$, for each $t \in [0, T]$ there exists such number $\delta(t)$, that under the condition $|\overline{t} - t| < \delta(t)$ the inclusion $\overline \partial h^2(x(\overline t),z(\overline t)) \subset B_r(\overline \partial h^2(x(t),z(t)))$ holds true at $\overline{t} \in [0, T]$, where~$r$ is some fixed finite positive number. The intervals $D_{\delta(t)}(t)$, $t \in [0, T]$, form an open cover of the segment $[0, T]$, so by Heine-Borel lemma one can take a finite subcover from this cover. Hence, there exists such number $\delta > 0$ that for every $t \in [0, T]$ the inclusion $\overline \partial h^2(x(\overline t),z(\overline t)) \subset$ $\subset B_r(\overline \partial h^2(x(t),z(t)))$ holds true once $|\overline{t} - t| < \delta$ and $\overline{t} \in [0, T]$. This means that for the segment $[0, T]$ there exists a finite partition $t_1 = 0, t_2, \dots, t_{N-1}, t_N = T$ with the diameter $\delta$ such that $\overline \partial h^2(x,z) \subset \bigcup\limits_{i=1}^N B_r(\overline \partial h^2(x(t_i),z(t_i)))$ for all $t \in [0, T]$. It remains to notice that the set $\bigcup\limits_{i=1}^N B_r(\overline \partial h^2(x(t_i),z(t_i)))$ is bounded due to the compactness of the set $\overline \partial h^2(x,z)$ at each fixed $t \in [0, T]$. 

As shown in statement 1) and at the beginning of statement 2) proof, the set $\overline \partial \varphi(x)$ is convex and its elements $w$ belong to the space $L_\infty^n [0, T]$. Then all the more the set $\overline \partial \varphi(x)$ is a convex subset of the space $L_2^n [0, T]$. Let us prove that the set $\overline \partial \varphi(x)$ is closed in the weak topology of the space $L_2^n [0, T]$. Let $\{w_n\}_{n=1}^{\infty}$ be the sequence of vector-functions from the set $\overline \partial \varphi(x)$, converging to the vector-function $w^*$ in the strong topology of the space $L_2^n [0, T]$. It is known \cite{Munroe} that this sequence has the subsequence $\{w_{n_k}\}_{n_k=1}^{\infty}$ converging pointwise to~$w^{*}$ almost everywhere on $[0, T]$, i. e. there exists such subset $T' \subset [0, T]$ having the measure $T$ that for every point $t \in T'$ we have $w_{n_k}(t) \in \overline \partial h^2(x(t),z(t))$ and $w_{n_k}(t)$ converges to $w^{*}(t)$, $n_k = 1, 2, \dots$. But the set $\overline \partial h^2(x(t),z(t))$ is closed at each $t \in [0, T]$ by the definition of the subdifferential, hence for every $t \in T'$ we have $w^{*}(t) \in \overline \partial h^2(x(t),z(t))$. So the set $\overline \partial \varphi(x)$ is closed in the strong topology of the space $L_2^n [0, T]$, but it is also convex, so it is also closed in the weak topology of the space $L_2^n [0, T]$ \cite{DunfordSchwartz}. 

Recall that by virtue of Remark 1 it is sufficient to consider the space $L_2^{n}[0, T]$. The weak* compactness of the set $\overline \partial \varphi(x)$ in the space $L_2^{n}[0, T]$ follows from its weak compactness (in $L_2^{n}[0, T]$) by virtue of these topologies definitions (see \cite{KolmFom}). The space $L_2^{n}[0, T]$ is reflexive \cite{DunfordSchwartz}, so the set there is weakly compact if and only if it is bounded in norm and weakly closed \cite{DunfordSchwartz} in this space. These properties required have been proved in the previous two paragraphs. The theorem is proved. 
\end{proof}

 \begin{remark}
An interesting problem for future research is to make an attempt to construct a numerical method for minimizing the functional $I(x, c)$ based on the well-known in nonsmooth optimization optimality conditions, which are written out using the formula for the superdifferential of the functional $\varphi(x)$ obtained in Theorem 3 (and with the formulas for the functionals $\chi(x)$, $\omega(x, c)$ G${\rm\hat{a}}$teaux gradients). 
\end{remark}

\section{Another type of control for providing an object motion in the vicinity of the surface $s(x) = 0_m$}

As is seen from the previous sections of the paper, the main difficulty while operating with control (\ref{5}) is caused by nondifferentiability (and even discontinuity) of this function in the phase variables. In order to overcome this difficulty, let us try to change the control structure in such a way that on the one hand, it would retain the main property of this control (namely, would ensure that the system hit some (acceptable in practice) neighborhood of the surface $s(x) = {\bf 0_m}$ and stay there), and on the other hand, would provide the desired control with the required differential properties, giving possibility (as we will see below) to use the developed ``variational'' technique to solve the ``full'' problem of finding the trajectory of the system both on the surface $s(x) = {\bf 0_m}$ and before hitting it. Note that in this section we do not use the term ``discontinuity surface'' (but simply talk about the surface $s(x) = {\bf 0_m}$) because, as will be shown below, the right-hand side of the system with the types of control structure considered in this section preserves the property of continuity (and even continuous differentiability if $x_i \neq 0$, $i = \overline{1,n}$).


In this section we restrict ourselves to consideration of the surfaces $s(x) = {\bf 0_m}$ with the hyperplanes $$s(x) = Cx - b,$$
where the elements of the $m \times n$ matrix $C$ and the vector $b$ of the dimension $m$ are still to be determined (we will omit the dependence of the surface vector-function on these variables for convenience of notation). Regulation of the surface structure via these parameters will give more opportunities to achieve certain goals of the object motion, for example, hitting the required position at the given time moment. As noted in section 3, in many practical cases it is natural to assume that these parameters are approximately known on the basis of the physical meaning of a system. Instead of control (\ref{5}) consider two more types of control for $i = \overline{1,m}$:  \begin{equation}
\label{7.1}
u_i^{[1]}(x, C, b) = -\alpha_i |x| s_i(x) \exp \Big\{\mathrm{sign}(s_i(x)) (-s_i(x)) \Big\}, 
\end{equation}

 \begin{equation}
\label{7.2}
\left\{
\begin{array}{lll}
u_i^{[2]}(x, C, b) = \alpha_i |x| k \sqrt{-s_i(x)}, \ &\text{if} \ s_i(x) \leqslant -\delta,  \\
u_i^{[2]}(x, C, b) = -\alpha_i |x| (e s^3_i(x) + f s_i(x)) , \ &\text{if} \ -\delta \leqslant s_i(x) \leqslant \delta, \\
u_i^{[2]}(x, C, b) = -\alpha_i |x| k \sqrt{s_i(x)}, \ &\text{if} \ s_i(x) \geqslant \delta,  
\end{array}
\right.
\end{equation}
where $\alpha_i$ are still some positive numbers which are fixed in this section (see Remark~10 below), the given number $\delta > 0$ is sufficiently small, and the parameters $k$, $e$ and $f$ are chosen based on the requirement of continuous differentiability of the functions $u_i^{[2]}(x, C, b)$ (if $x_i \neq 0$, $i = \overline{1,n}$). It is not difficult to check that for every given in advance vicinity of the surface $s(x) =~{\bf 0_m}$ one can ensure the object hitting this vicinity in a finite time via the choice of the numbers $\alpha_i$, $i = \overline{1,m}$, and $\delta$. The disadvantage of such a choice of controls is as follows: despite the fact that they transfer the system to a small neighborhood of the chosen surface (that is, in accordance with the definition given in Introduction, they ensure the stability of the sliding mode), these controls do not endow the system with the stability property ``in big'', since the value of $ s(x(t))$ generally does not tend to zero as $t$ tends to infinity. Hence, ``correction'' of the surface parameters has an additional sense: to more strictly satisfy the conditions on the right endpoint (than the accuracy provided by introduced controls (\ref{7.1}), (\ref{7.2})). The advantage of such a choice is the fact that the right-hand side of the system is a continuously differentiable function of the phase coordinates if $x_i \neq 0$, $i = \overline{1,n}$ (and of the variables $C$, $b$). Note that although controls used in works \cite{Tang}, \cite{Shtessel} (mentioned in Introduction) are continuous, they are not continuously differentiable in the phase coordinates.

Let us also note the constructive features of the controls given. The advantage of the control $u^{[1]}$ is the relative simplicity (compared to the control~$u^{[2]}$) of defining this function. The disadvantage of the control $u^{[1]}$ is that the necessary (in order to ensure the stability of the sliding mode) values of the gains~$\alpha_i$, $i = \overline{1,m}$, are too large and are even sometimes unacceptable for practice in the case when the initial position of the object is rather far from the surface $s(x) = {\bf 0_m}$; and for any gains~$\alpha_i$, $i = \overline{1,m}$, taken in advance, there exists an initial object point from which control $u^{[1]}$ is unable to bring the system to the small vicinity of the surface $s(x) = {\bf 0_m}$ (the control $u^{[2]}$ control does not have this negative feature). 

As previously, the problem is of finding a trajectory $x^*$ from the space $C_n[0,T]$ (with the derivative $\dot x^*$ from the space $P_n[0,T]$) described by the system  
$$
\dot x_i = A_i x + u_i^{[1], [2]} (x, C, b), \quad i = \overline{1, m}, 
$$
$$
\dot x_i = A_i x, \quad i = \overline{m+1, n},
$$
with new control (\ref{7.1}) or (\ref{7.2}) (the parameters $C^* \in R^m \times R^n$ and $b^* \in R^m$ are to be determined as well), satisfying boundary conditions (\ref{2}), (\ref{3}). We assume that there exists such a solution.

Denote
$$h_i^{[1],[2]}(x, z, C, b) = z_i - A_i x - u_i^{[1], [2]}(x, C, b), \quad i = \overline{1,m}, $$ 
$$ h_i(x, z) = z_i - A_i x, \quad i = \overline{m+1,n},$$  
and construct the functional
$$\varphi^{[1],[2]}(z, C, b) =  \frac{1}{2} \sum_{i=1}^m \int_0^T h_i^{{2 \, [1],[2]}}(x(t), z(t), C, b) dt + \frac{1}{2} \sum_{i=m+1}^n \int_0^T h_i^2(x(t), z(t)) dt, $$
where instead of the phase variable $x(t)$ one should write its expression via its derivative $z(t)$ by formula  (\ref{3.11''}). 

Then as previously construct the functional  
\begin{equation}
\label{7.4}
I^{[1],[2]}(z, C, b) = \varphi^{[1],[2]}(z, C, b) + \chi(z) ,
\end{equation}
in which the function $\chi(z)$ is given by formula (\ref{3.14}).

Consider the problem of minimizing functional (\ref{7.4}) on the space \linebreak $P_n [0, T] \times R^{m} \times R^n \times R^m$. Denote $(z^*, C^*, b^*)$ a global minimizer of this functional. Then $$ x^*(t) = x_0 + \int_0^t z^*(\tau) d\tau $$
is a solution of the original problem (and the matrix $C^*$ and the vector $b^*$ define the considered surface structure). 

It is obvious that the point $x^*$ is a solution of Problem 2 iff the functional $I^{[1],[2]}(z, C, b)$ vanishes at the corresponding point, i.~e. $I^{[1],[2]}(z^*,C^*, b^*) = 0$. In order to obtain a more constructive minimum condition, which is useful for developing a numerical method for solving the original problem, let us study the differential properties of the functional $I^{[1],[2]}(z, C, b)$. Suppose that the trajectories $x_i(t)$, $i = \overline{1,n}$, vanish only at isolated time moments of the segment~$[0, T]$. This assumption is natural if before the transition to the sliding mode there are no trajectories which are in zero position and have zero ``speed'' on some subspace of the interval $[0, T]$ of nonzero measure and if the discontinuity surfaces do not contain any of these trajectories vanishing on the interval $[0, T]$ subset of nonzero measure. With the help of the classical variation, Lagrange's mean value theorem and integration by parts it is easy to check the G${\rm\hat{a}}$teaux differentiability of the given functional. 

\begin{theorem}
\label{th:4}
Let the trajectories $x_i(t)$, $i = \overline{1,n}$, vanish only at isolated time moments of the interval $[0, T]$. Then the functional $I^{[1], [2]}(z, C, b)$ is G${\rm\hat{a}}$teaux differentiable and its gradient at the point $(z, C, b)$ is expressed by the formula
$$\nabla I^{[1],[2]}(z, C, b) =  \sum_{i=1}^m \Bigg[ h_i^{[1],[2]}(x(t), z(t), C, b) {\bf e_i} - $$ $$-\int_t^T h_i^{[1],[2]}(x(\tau), z(\tau), C, b) \left(A'_i+\frac{\partial u_i^{[1],[2]}(x(\tau), C, b) }{\partial x} \right) d \tau,$$
$$-\int_0^T h_i^{[1],[2]}(x(t), z(t), C, b) \frac{\partial u_i^{[1],[2]}(x(t), C, b) }{\partial C} dt,  -$$$$-\int_0^T h_i^{[1],[2]}(x(t), z(t), C, b) \frac{\partial u_i^{[1],[2]}(x(t), C, b) }{\partial b} dt \Bigg] + $$
$$+  \Bigg[ \sum_{i=m+1}^n h_i(x(t), z(t)) {\bf e_i} - \int_t^T h_i(x(\tau), z(\tau)) A'_i \, d\tau, \, {\bf 0_{m \times n}}, \, {\bf 0_m}  \Bigg] +$$
$$+
\Bigg[ \sum_{j \in J} \left( {x_0}_j + \int_0^T z_j(t) dt - {x_T}_j \right) {\bf e_j}, \, {\bf 0_{m \times n}}, \, {\bf 0_m}  \Bigg] 
.$$
\end{theorem}

Let us write down the known minimum condition for a Gateaux differentiable functional. 

\begin{theorem}
\label{th:5}
{Let the trajectories $x_i(t)$, $i = \overline{1,n}$, vanish only at the isolated time moments of the interval $[0, T]$. In order for the point $(z^*,C^*, b^*)$ to minimize the functional $I^{[1],[2]}(z, C, b)$, it is necessary that
$$
0_n \times {\bf 0_n} \times {\bf 0_m} \times {\bf 0_m}  =  \sum_{i=1}^m \Bigg[ h_i^{[1],[2]}(x^*(t), z^*(t), C^*, b^*) {\bf e_i} - $$ $$-\int_t^T h_i^{[1],[2]}(x^*(\tau), z^*(\tau), C^*, b^*) \left(A'_i+\frac{\partial u_i^{[1],[2]}(x^*(\tau), C^*, b^*) }{\partial x} \right) d \tau,  
$$
$$-\int_0^T h_i^{[1],[2]}(x^*(t), z^*(t), C^*, b^*) \frac{\partial u_i^{[1],[2]}(x^*(t), C^*, b^*) }{\partial C} dt, -$$ $$-\int_0^T h_i^{[1],[2]}(x^*(t), z^*(t), C^*, b^*) \frac{\partial u_i^{[1],[2]}(x^*(t), C^*, b^*) }{\partial b} dt \Bigg] + $$
$$ +  \Bigg[ \sum_{i=m+1}^n h_i(x^*(t), z^*(t)) {\bf e_i} - \int_t^T h_i(x^*(\tau), z^*(\tau)) A'_i \, d\tau, \, {\bf 0_{m \times n}}, \, {\bf 0_m}  \Bigg]+$$
$$+
\Bigg[ \sum_{j \in J} \left( {x_0}_j + \int_0^T z^*_j(t) dt - {x_T}_j \right) {\bf e_j}, \, {\bf 0_{m \times n}}, \, {\bf 0_m}  \Bigg]. $$
where $0_n$ is a zero element of the space $P_n[0, T]$.
}
\end{theorem}

It is apparent, that $I(z^*, C^*, b^*) = 0$ is necessary and sufficient minimum condition for this functional and in this case the equality of Theorem 5 is automatically satisfied.

{\bf Example 2.}
Consider the following system
$$\dot x_1(t) = x_1(t) - 3 x_2(t) + x_3(t) - 50 |x(t)| (c_{11} x_1(t) - b_1) \exp \Big\{ (-|c_{11} x_1(t) - b_1|) \Big\},$$
$$\dot x_2(t) = 5 x_1(t) + x_2(t) - x_3(t) - 50 |x(t)| (c_{22} x_2(t) - b_2) \exp \Big\{(-|c_{22} x_2(t) - b_2|) \Big\},$$ 
$$\dot x_3(t) = 5 x_1(t) - x_2(t) + x_3(t)$$
with the boundary conditions
$$x_1(0) = 2, \ x_2(0) = -2, \ x_3(0) = 2, $$
$$x_1(0.2) = 0.55, \ x_2(0.2) = 2.5, \ x_3(0.2) = 2.95.$$
As we see, the surfaces in this example are of the form
$$s_1(x) = c_{11} x_1 - b_1 = 0, \quad  s_2(x) = c_{22} x_2 - b_2 = 0.$$

A simplest steepest descent method (in the functional space) \cite{Kantorovich} with some modifications (see Remark 7, b) above) was used in order to minimize the functional $I^{[1]}(z, c, b)$ in this problem. We use the notation for the parameters $c = (c_{11}, c_{22})'$ here instead of the matrix $C$ for simplicity. 

The point $(z_{\{0\}}, c_{\{0\}}, b_{\{0\}}) = (0, 0, 0, 0.18, 0.2, 0.12, 0.51)'$ was taken as the initial one (note that here $(0, 0, 0)'$ is a zero point in the functional space $P_3 [0,1]$ and the points $(0.18, 0.2)'$ and $(0.12, 0.51)'$ belong to the spaces $R^2$ and $R^2$ respectively). Herewith, we have $I(z_{\{0\}}, c_{\{0\}}, b_{\{0\}}) \approx 1591.75905$. Substitute the parameters $c_{\{0\}}, b_{\{0\}}$ into the system given. Integrate the closed-loop system numerically (for this example the Rosenbrock stiff 3-4-th order method was used) and obtain $x_1(0.2) \approx 0.5879$, $x_2(0.2) \approx 2.59282$, $x_3(0.2) \approx 2.97868$; we see that the desired value on the right endpoint is achieved with an error of the order $10^{-1}$.  

At the 79-th iteration the point $(z_{\{79\}}, c_{\{79\}}, b_{\{79\}})' $ was constructed and we we approximately put  $(z^*, c^*, b^*) = (z_{\{79\}}, c_{\{79\}}, b_{\{79\}})$. Herewith \linebreak $c^* \approx (0.1836729, 0.2016907)'$, $b^* \approx (0.1139969, 0.4974675)'$,  $I(z^*, c^*, b^*) \approx 0.00064$ and $|| \nabla I(z^*, c^*, b^*) ||_{L^2_3 [0,1] \times R^2 \times R^2} \approx 0.0308$. The point $z^*$ is not given here for two reasons: 1) for brevity (as it is constructed as a rather bulky piecewise continuous vector-function) and 2) because only the parameters $c^*$ and $b^*$ are finally used to estimate the result obtained (see the next paragraph).  

Having obtained the parameters $c^*$, $b^*$, substitute them into the system given. Integrate the closed-loop system with the Rosenbrock stiff \linebreak 3-4-th order method and finally have $x_1(0.2) \approx 0.54513$, $x_2(0.2) \approx 2.50537$, \linebreak $x_3(0.2) \approx 2.95221$, so we see that the desired value on the right endpoint is achieved with an error of the order $5 \times 10^{-3}$. So these values have been improved via ``correcting'' the parameters of the surface considered. Picture~2 demonstrates the results of calculations. The black lines denote the curves which were obtained via the method of the paper, while the dashed red lines denote the curves obtained via integration of the closed-loop system. It is easy to see that $|s_1(x^*(t))|, |s_2(x^*(t)|  < \varepsilon_s$ beginning from some time moment, where the value $\varepsilon_s$ is of the order $10^{-1}$, so the system considered hits \linebreak $\varepsilon_s$-vicinity of the desired surface at some finite time moment and remains there after this moment. For decreasing this vicinity, one has to increase the gains~$\alpha_i$, $i = \overline{1,m}$ (see also Remark 10 below).

 \begin{figure*}[h!]
\begin{minipage}[h]{0.3\linewidth}
\center{\includegraphics[width=1\linewidth]{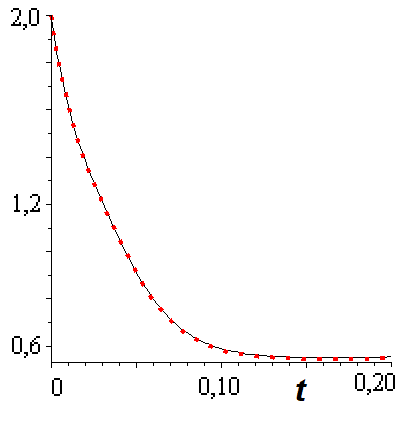} }
\end{minipage}
\hfill
\begin{minipage}[h]{0.3\linewidth}
\center{\includegraphics[width=1\linewidth]{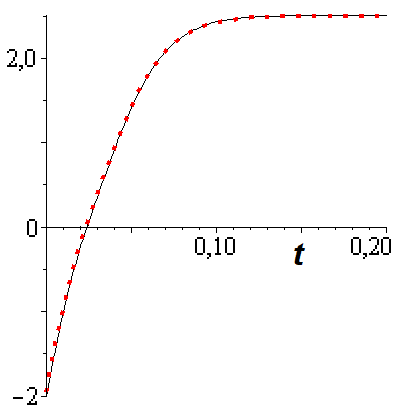} }
\end{minipage}
\hfill
\begin{minipage}[h]{0.3\linewidth}
\center{\includegraphics[width=1\linewidth]{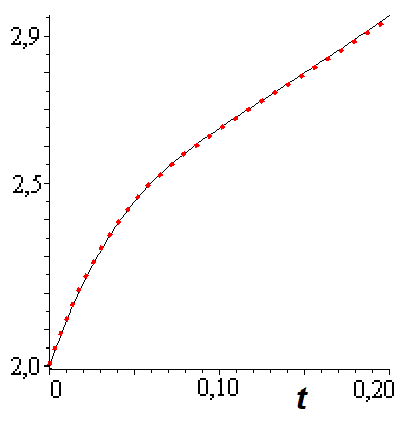} }
\end{minipage}
\caption{Example 2. The trajectories $x_1^*(t)$, $x_2^*(t)$, $x_3^*(t)$.}
\end{figure*}

\begin{remark}
In addition to regulating the structure of the vector-function $s(x)$ (i.e. the parameters $C$ and $b$), one may consider the gains $\alpha_i$, $i = \overline{1,m}$, ``free'' as well and optimize them based on the desired properties of the system according to a similar scheme. 
\end{remark}


\begin{remark}
As the paper presented is at most of a theoretical nature, \linebreak Examples~1, 2 are rather simple and aim at only illustrating the approach developed, so we don't describe the computational aspects in detail. It is planned to consider more complicated examples and to study some aspects of the method, as well as to use its modifications (taking into account the specifics of the functionals under consideration) or some other known optimization methods in functional space in order to improve computational efficiency, in future research.  
\end{remark}

\end{document}